\documentclass[a4paper]{easychair}

\usepackage{comment}
\usepackage{url}
\usepackage{hyperref}
\usepackage{graphicx}
\usepackage[utf8]{inputenc}            
\usepackage{amssymb,latexsym, dsfont, amsfonts, amsbsy, amsmath, mathrsfs, dsfont}            
\usepackage{mathtools}                   
\usepackage{bm}                          
\usepackage{enumerate}                   
\usepackage{verbatim}                    
\usepackage{url}                         

\usepackage{color}
\usepackage{tabularx}
\usepackage{microtype}
\usepackage{latexsym}
\usepackage{amsmath,amssymb,amscd}
\usepackage{graphicx}
\usepackage{tikz}
\usepackage{enumitem}

\usepackage[all,cmtip]{xy}

\newtheorem{theorem}{Theorem}[section]
\newtheorem{corollary}[theorem]{Corollary}
\newtheorem{lemma}[theorem]{Lemma}

\newtheorem{definition}[theorem]{Definition}
\newtheorem{remark}[theorem]{Remark}

\newtheorem{observation}[theorem]{Observation}

\newcommand{\Go}[1]{{G\"{o}del} }

\usepackage{makecell}
\setcellgapes{2pt}

\usepackage{color}

\DeclareMathSymbol{\sm}{\mathbin}{AMSa}{"39}

\newcommand{\class}[1]{\mathds{#1}}
\renewcommand{\mod}[1]{\mathfrak{#1}}

\begin{document}
\setcounter{page}{1}     




\title{Axiomatizing logics of finite G\"odel-Kripke models}
\author{Amanda Vidal$^1$ \\ amanda@cas.cas.cz  \and \\ Ricardo O. Rodriguez$^2$ \\ ricardo@dc.uba.ar}

\institute{
Institute of Computer Science of the Czech Academy of Sciences \and  University of Buenos Aires, 
Buenos Aires, Argentina
}

	\maketitle
	
\titlerunning{Axiomatizing logics of finite G\"odel-Kripke models}
\authorrunning{A. Vidal, R.O. Rodriguez}

\begin{abstract}

We investigate completeness for modal G\"odel logics with respect to finite G\"odel-Kripke models, along with related aspects. It is well known that the logics studied in \cite{CaRo10, MeOl09} fail to be complete with respect to finite G\"odel-Kripke models. We show that the natural candidate axiomatic extensions do not restore completeness, thereby resolving a 15 year open problem posed in the aforementioned works. We then provide new axiomatizations that are complete for finite models and characterize intermediate witnessing conditions that hold for the basic logics.


\end{abstract}

\section{Introduction} 

Modal G\"odel logics constitute a robust and mathematically rich family of many-valued systems. They arise naturally by extending G\"odel logic (the  semilinear extension of intuitionistic logic) with modal operators that generalize the standard Kripkean framework to a fuzzy or graded setting. These logics represent a synthesis of many-valued reasoning and modality, capturing nuanced, graded notions of necessity ($\Box$) and possibility ($\Diamond$) while preserving the fundamental algebraic and order-theoretic properties of G\"odel structures.

It is known that the logic defined as the local entailment over classes of Kripke models valued over the standard Gödel algebra coincides with the one resulting from adding semilinearity to Fischer-Servi Intuitionistic modal logic \cite{FS84b} KIC. The former perspective, which we address informally as ``valued Kripke semantics" is naturally grounded in the interaction between Kripke-style semantics and propositional algebras. The latter, ``purely relational semantics", yields a semantics based on purely relational models (with two relations: one governing the propositional connectives and the second handling the modal operators), via the completeness of KIC from \cite{FS84b}. Both semantics can be interpreted in the more general algebraic semantics, via the construction of the corresponding complex algebras, losing, however, intuition and relational appeal.

The central focus of the present work is the divergence between crucial properties of modal G\"odel logics in comparing the above semantics. It is established that the basic logic $K$ possesses the Finite Model Property (FMP)\cite{CaMeRo17}, so it enjoys this property with respect to its algebraic semantics of modal G\"odel algebras. Consequently, any non-theorem is refutable in some finite algebraic structure. However, this ``finitary" behavior fails to translate to the intended G\"odel-Kripke semantics over the standard interval $[0,1]$. It is proven that there exist formulas valid in all finite $[0,1]$-valued Kripke models that are not theorems of the logic \cite{CaRo10,MeOl09}. The phenomenon persists even when restricting the scope to crisp (classical) accessibility relations or mono-modal fragments, suggesting that the ``finiteness" of a world-set imposes logical constraints whose meaning in the algebraic semantics is necessarily different.

This paper is motivated by the problem of characterizing the logical entailment induced by finite valued relational models. In this context, we investigate the phenomena of \textit{witnessing} and \textit{unwitnessing}. In an infinite G\"odel-Kripke model, the truth value of $\Diamond \varphi$ and $\Box \varphi$ is defined as a supremum or infimum over accessible states, which need not be attained at any particular world. By contrast, every finite model is necessarily \textit{witnessed}, since every finite subset of $[0,1]$ has a maximum and a minimum. From this perspective, the logic of finite models can be understood as capturing modalities based on attained suprema and infima.

From a mathematical point of view, this leads to a more precise understanding of the unwitnessing phenomena in G\"odel modal logics. In particular, we settle an open problem posed in \cite{CaRo10} (see the end of p.~198), as well as the analogous question arising from \cite{MeOl09}, by showing that the formulas previously used to establish the failure of the finite model property (with respect to this semantics) are not sufficient to axiomatize the corresponding finite semantics. We then introduce complete axiomatizations for the logics under study. In complementary results we also clarify interesting intermediate conditions, obtaining completeness results for the basic modal logics  with respect to classes of models associated with partial witnessing conditions.

The results developed here also suggest several further directions and applications. First, the methods introduced in this paper are expected to extend naturally to G\"odel $K45$ and $KD45$ logics, both in crisp and valued semantics \cite{BoEsGoRo16}, potentially yielding more transparent semantic characterizations for these systems. Moreover, since modal logics may be viewed as fragments of first-order logics, the present results may also be relevant in connection with the developments of \cite{BI-IGPL2016, IEMH10}. Finally, our work complements the proof-theoretic investigations of \cite{FeFiRo25} and \cite{FeFiGiRo26}, where cut-free sequent calculi for $\class{KG}^\omega$ and $\class{MG}^\omega$ are introduced. While those calculi provide refutation and decision procedures, they do not constitute deductive systems for effectively generating the theorems of the corresponding logics.

The main contributions of this paper are the following ones. We propose axiomatic systems that are proved complete with respect to the corresponding classes of finite G\"odel-Kripke models, both for valued accessibility (Theorem \ref{th:completenessVal}) and crisp accessibility (Theorem \ref{th:completenessCrisp}). Notably, our axioms utilize isolated modalities, thereby providing also the first axiomatizations for the mono-modal fragments of these finite-model logics. We also demonstrate that the formulas from \cite{CaRo10, MeOl09} are incomplete for the class of finite models (Lemmas \ref{lem:NotWitBox} and \ref{lem:NotWitDiamond}), proving that our new approach is necessary.
Furthermore, we establish intermediate results of independent interest: although the foundational logics lack the finite model property (FMP), they retain completeness with respect to models that satisfy specific \textit{partial witnessing} conditions. Specifically, we prove that the minimal bi-modal G\"odel logic is complete for the class of $\Diamond$-witnessed models (Theorem \ref{th:valueddiamondcompleteness}), and that crisp logics satisfy a weaker but highly useful condition we term $\langle \Diamond, 1 \rangle$-witnessing (Theorem \ref{th:topwitnessed}).



\section{Preliminaries, known results and open problems}

We begin by fixing notation and recalling some basic facts about modal G\"odel logics that will be used throughout the paper.

\begin{definition}
A \emph{G\"odel-Kripke model} is a tuple $\langle W, R, e\rangle$ where $W$ is a non-empty set, $R\colon W \times W \rightarrow [0,1]$ and $e\colon W \times \mathcal{V} \rightarrow [0,1]$, such that $e$ is, world-wise, a propositional G\"odel homomorphism (from the algebra of formulas into $[0,1]$, and 
\[e(v,\Box \varphi) \coloneqq \bigwedge_{w \in W}R(v,w) \rightarrow e(w, \varphi) \qquad, \qquad e(v,\Diamond \varphi) \coloneqq \bigvee_{w \in W}R(v,w) \wedge e(w, \varphi).\]
\end{definition}
If $R(v,w) \in \{0,1\}$ for any $v,w \in W$, we say that the model is \emph{crisp}. A model is \emph{witnessed} whenever for each $\Box \varphi, \Diamond \varphi$,  
\[e(v,\Box \varphi) \coloneqq R(v,v_{\Box \varphi}) \rightarrow e(v_{\Box \varphi}, \varphi), \qquad e(v,\Diamond \varphi) \coloneqq R(v,v_{\Diamond \varphi}) \wedge e(v_{\Diamond \varphi}, \varphi)\] for some $v_{\Box \varphi}, v_{\Diamond \varphi} \in W$. 
We denote by $\class{MG}$ the full class of G\"odel-Kripke models, by $\class{KG}$ the class of crisp G\"odel-Kripke models and by $\class{MG}^{\omega}$, $\class{KG}^{\omega}$ the subclasses of witnessed models of the previous two.


Given a model $\mod{M}$, we say that $\Gamma \Vdash_{\mod{M}} \varphi$ whenever, for any $v \in W$, if $e(v, \gamma) = 1$ for any $\gamma \in \Gamma$, then $e(v, \varphi) = 1$ too.\footnote{The logical systems defined and studied in this paper are known as \emph{local} modal logics, as opposed to \emph{global} modal logics, which we do not address here. In the corresponding axiomatic systems, the fact that inference rules apply only to theorems is the key difference between local and global logics.} Note that $\Gamma, \varphi$ can be formulas in both $\Box, \Diamond$ modal operators, or only in one of them. It is known that these symbols are not interdefinable. Consequently, we will let $Fm$ be the set of bi-modal formulas, in both $\Box, \Diamond$, and $Fm_{\Box}/Fm_{\Diamond}$ the sets of mono-modal formulas (written in a single modality).

Given a class of models $\class{C}$, we say that a logic $L$ (in a modal language) is complete with respect to $\class{C}$  (or arises from $\class{C}$) whenever for any $\Gamma \cup \{\varphi\}$, 
\[\Gamma \vdash_{L} \varphi \Longleftrightarrow \forall \mod{M} \in \class{C}, v \in W,  (\forall \gamma \in  \Gamma :  e(v, \gamma) = 1) \text{ implies } e(v, \varphi) = 1.\]
Observe the previous definition will allow us to say that logics built over different modal languages -namely, for the formulas $Fm$, $Fm_\Box$ and $Fm_\Diamond$- are all complete with respect to a same class of models.

The \textit{bi-modal} logic of $\class{MG}$ was axiomatized in \cite{CaRo15}, and we will denote it by $MG$. The logic of $\class{KG}$ over the full bi-modal language was axiomatized in \cite{RoVi20}, and we will denote it by $KG$. The corresponding mono-modal fragments, which we will denote by $(MG/KG)_{(\Box/\Diamond)}$ were axiomatized in \cite{CaRo10} and \cite{MeOl09}. It is known that $MG_\Box = KG_\Box$.

We provide the details of the previous axiomatic systems here for the reader's convenience.
 The logics $MG_{\Box }$ and $MG_{\Diamond }$ are axiomatized by adding to G\"odel-Dummett\ propositional calculus the following axiom schemes and inference rules:
\begin{equation*}
\begin{array}{ll}
& \hspace{1cm} {MG}_{\Box }\\
{K}_\Box: & \Box (\varphi \rightarrow \psi)\rightarrow (\Box \varphi \rightarrow \Box \psi ) \hspace{1cm} \\
{Z}_\Box: & \lnot \lnot \Box \varphi \rightarrow \Box \lnot \lnot \varphi \\
{N}_\Box: & \vdash \varphi \emph{\ implies \ } \vdash \Box \varphi \text{ } \\
&
\end{array}%
\text{ \ }%
\begin{array}{ll}
& \hspace{1cm} {MG}_{\Diamond } \\
{K}_\Diamond: & \Diamond (\varphi \vee \psi ) \rightarrow (\Diamond \varphi \vee \Diamond \psi ) \\
{Z}_\Diamond: &  \Diamond \lnot \lnot \varphi \rightarrow \lnot \lnot \Diamond \varphi \\
{F}_\Diamond: &  \lnot \Diamond \bot \\
{N}_\Diamond: &  \vdash \varphi \rightarrow \psi \emph{\  implies \ }\vdash \Diamond \varphi
\rightarrow \Diamond \psi .%
\end{array}%
\end{equation*}

The system $KG_\Diamond$ results by replacing in the axiomatization of $MG_\Diamond$ rule ${N}_\Diamond$ by the following one:
\[{N}^c_\Diamond\colon   \vdash (\varphi \rightarrow \psi) \vee \chi \text{\emph{\ implies \ }} \vdash(\Diamond \varphi \rightarrow \Diamond \psi) \vee \Diamond \chi. \]

The  system ${MG}$ (i.e., the bi-modal godel logic) results from adding to the union of ${MG}_{\Box }$ and ${MG}_{\Diamond }$, Fischer Servi's connecting axioms \cite{FS84b}:

\[
{FS}_1 \colon \Diamond (\varphi \rightarrow \psi )\rightarrow (\Box \varphi \rightarrow
\Diamond \psi ), \qquad 
 {FS}_2\colon (\Diamond \varphi \rightarrow \Box \psi )\rightarrow \Box (\varphi
\rightarrow \psi ).
\]

Finally, the system $KG$ is obtained by adding to $MG$ the so-called Dunn's axiom:
\[ {Cr}\colon  \Box(\varphi \vee \psi) \rightarrow (\Box \varphi \vee \Diamond \psi).\]

It is not hard to see that the above logics, and those arising from $\class{MG}^\omega$ and $\class{KG}^\omega$, enjoy completeness also with respect to trees of bounded finite depth:
namely, structures for which the relation $R$ in values strictly greater than $0$ gives the model the shape of a tree (of possibly infinite width) and each branch is of depth bounded by the modal depth of the formulas in the derivation.\footnote{The idea is similar to the one used in classical (local) modal logic. For a detailed formal proof of the fact that the basic local logics are complete with respect to models (trees) of finite depth see for instance Proposition 2.7 and Corollary 2.8 in \cite{ViManuscript}, that apply equally to G\"odel logic. We emphasize that this holds because we are considering the \textit{local} entailment.} This directly implies the following observation.

\begin{observation}\label{obs:finDepthDecidable}
    The logics arising from $\class{MG}^\omega$ and from $\class{KG}^\omega$ are complete with respect to finite models of bounded depth, and so, decidable.\footnote{In more detail, the maximum size of a (possible) counter-model for a formula $\varphi$ of modal depth $n$, in $\class{MG}^\omega$ or $\class{KG}^\omega$ is given by $1 + \sum_{i = 0}^{i=n} \prod_{j\leq i }\vert \varphi^j\vert$ for $\varphi^0 \coloneqq PSFm(\varphi), \varphi^{k+1} \coloneqq PSFm(\bigwedge \varphi^k)$, and $PSFm(x) \coloneqq \emptyset, PSFm(\psi * \chi) \coloneqq PSFm(\psi) \cup PSFm(\chi) \text{ for }*\text{ propositional connective }, PSFm( \star \psi) \coloneqq \{\psi\} \text{ for }\star\text{ modal connnective}.$ This bound comes from a model with a root and, and adding at each level of modal depth less than that of $\varphi$ (i.e., that $n$), as many new worlds as to witness each possible subformula at that depth.
    }
\end{observation}

A central theme in the study of these logics is the interaction between completeness, decidability, and witnessing conditions. Several results in the literature address this relationship. On the one hand, it is known that  $MG_\Diamond$ is complete with respect to $\class{MG}^\omega$, and so, it is decidable by Observation \ref{obs:finDepthDecidable} above. On the other hand, the full bi-modal logics $MG$ and $KG$, as well as the box fragments $(MG/KG)_\Box$, fail to be complete with respect to their classes of witnessed models \cite{MeOl09,CaRo10} (and hence do not enjoy the FMP with respect to our intended semantics).
The following formulas, introduced in the previous references, give counterexamples witnessing these failures: formula $UW_\Box$ for the $\Box$ fragments and for $KG$, and $UW_\Diamond$ for $KG_\Diamond$:
\[
  UW_\Box\colon  \Box(\neg \neg x) \rightarrow \neg \neg \Box x, \qquad UW_\Diamond\colon (\Diamond x \rightarrow \Diamond y) \rightarrow ((\neg \Diamond y \vee \Diamond(x \rightarrow y))).
\]

Interestingly enough, it was first proven via proof-theoretic methods that $KG_\Box$ and $KG_\Diamond$ are decidable and PSPACE-complete \cite{MeOl11}, and later on, that the full $MG$ and $KG$ logics are also decidable \cite{CaMeRo17} via an hybrid semantics.

Some results addressing partial witnessing conditions have also been presented in the literature. For instance, the extension $(MG/KG)_\Box + UW_\Box$ has been proven to be complete with respect to the class of $\langle \Box, 0\rangle$-witnessed models \cite[p.~198]{CaRo10}.\footnote{Namely, models in which $e(v,\Box \varphi)=0$ implies there is a world $u$ such that $e(u,\varphi)=0$ and $R(v,u)>0$.} This illustrates how relatively simple syntactic principles can enforce partial witnessing.

Despite these advances, several fundamental questions remain open, as we highlighted in the introduction. It was proposed (and remains) as an open question in \cite{CaRo10} whether the logic $(MG/KG)_\Box + UW_\Box$ enjoys the FMP (namely, whether it it is complete with respect to $\class{KG}^\omega$), and it is natural to address the analogous question for the logic $KG_\Diamond + UW_\Diamond$. Furthermore, and in tight relation to the previous questions, remains the problem of providing systematic axiomatizations for the various witnessed logics.


In a related line of research, it is natural to ask whether some of the logics mentioned above admit \emph{partial} witnessing conditions in a meaningful sense -for instance, being witnessed only for $\Diamond$, only for $\Box$, or satisfying weaker “quasi-witnessed” requirements. 

In the rest of the paper we will address the previous questions. First, in Section \ref{sec:failwit} by answering negatively to the open questions concerning the witnessing completeness of the natural $(MG/KG)_\Box + UW_\Box$ and $KG_\Diamond + UW_\Diamond$ logics. In Section \ref{sec:partialWitnessing} we will show several partial witnessing conditions that apply to $\class{MG}$ and $\class{KG}$. Lastly, and relying in the insights obtained in the previous sections, we will propose axiomatic systems and prove their completeness with respect to $\class{MG}^\omega$ and $\class{KG}^\omega$, also working in the mono-modal frameworks.

\section{On the failure of completeness with respect to finite G\"odel Kripke models}\label{sec:failwit}

In this section, we establish two negative results concerning completeness with respect to finite G\"odel–Kripke models. In particular, we show that the extensions of $KG_\Box$ and $KG_\Diamond$ respectively by the formulas $UW_\Box$ and $UW_\Diamond$, proposed in \cite{CaRo10} and \cite{MeOl09} respectively, still fail to achieve completeness over the class of finite G\"odel–Kripke models. This settles a problem that has remained open since 2010, resolving the questions posed in \cite{CaRo10} and \cite{MeOl09}.

To this end, we introduce new formulas that are valid in all finite G\"odel–Kripke models, yet not derivable in the corresponding extended logics. For clarity, we recall that in the standard G\"odel algebra, for any $x,y\in [0,1]$,
\[(x \rightarrow y) \rightarrow y \coloneqq \begin{cases} 1 &\hbox{ iff } y < x \text{ or } y = 1,\\ y &\hbox{ otherwise. } \end{cases}\]

Focusing on the $\Box$-fragment, it is not hard to pinpoint where the formula $UW_\Box$ fails to adequately capture unwitnessing phenomena: namely, at intermediate truth values distinct from $0$. This observation suggests the need for a formula forcing the fact that, if the values of two formulas differ at every successor world, then necessarily also  their $\Box$-evaluations coincide. The following lemma formalizes this intuition.

\begin{lemma}\label{lem:NotWitBox}
$(MG/KG)_\Box + \Box \neg \neg x \rightarrow \neg \neg \Box x$  
is not complete with respect to witnessed models.
\end{lemma}
\begin{proof}
We will denote by $KG_\Diamond^\star$ to the logic $KG_\Box + \Box \neg \neg x \rightarrow \neg \neg \Box x$. 
Consider the formula 
\begin{equation}\label{eq:NotWitBox}
\Box (((x \rightarrow y) \rightarrow y) \wedge (y \rightarrow x)) \rightarrow ((\Box x \rightarrow \Box y) \rightarrow \Box y).    
\end{equation}


Let us first check that this formula is sound with respect to witnessed models. Since $MG_\Box = KG_\Box$, we can simply do this for crisp models. 
 Take a crisp witnessed model $\mathfrak{M}$ and a world $u$ in it.
The only non trivial case to check is when $e(u, \Box x) \leq e(u, \Box y) < 1$, so the value of the right side of the implication (\ref{eq:NotWitBox}) equals $e(u, \Box y)$ and strictly less than $1$. For $v$ such that $Ruv$ and $e(u, \Box x) = e(v,x)$ (which exists because the model is witnessed), we have that $e(v,x) \leq e(u, \Box y) \leq e(v,y)$. Let us consider two cases. If $e(v,x) < e(v,y)$, it follows that $e(v, (((x \rightarrow y) \rightarrow y) \wedge (y \rightarrow x)) = e(v, y \wedge x) = e(v,x)$. Since we know that $e(v, x) = e(u, \Box x) \leq  e(u, \Box y)$, this makes the implication (\ref{eq:NotWitBox}) true. 
For the second case, suppose that $e(v,x) = e(v,y)$, and observe this implies that $e(u, \Box y) = e(v,y)$ too. 
Then, $e(v, (((x \rightarrow y) \rightarrow y) \wedge (y \rightarrow x)) = e(v, y \wedge x) = e(v,y)$, satisfying the implication (\ref{eq:NotWitBox}) using the previous observation.

We will now produce a $KG_\Box^\star$ model that falsifies the previous formula. 

Consider the model with universe $u \cup \{v_i\colon i \in \omega, i > 1\}$, $Ruv_i$ for all $i$, and $e(v_i, x) = \frac{1}{2}+\frac{1}{i}$, $e(v_i, y) = \frac{1}{2}$ (the values taken at $u$ are irrelevant). It is clear that $e(v_i, x) > e(v_i, y)$ for any $i$, so $e(v_i, (((x \rightarrow y) \rightarrow y) \wedge (y \rightarrow x))= 1$ for any $i$, and therefore, $e(u, \Box ((((x \rightarrow y) \rightarrow y) \wedge (y \rightarrow x)) = 1$. On the other hand, since $e(u, \Box x) = e(u, \Box y) = \frac{1}{2}$, we get  $e(u, (\Box x \rightarrow \Box y) \rightarrow \Box y) = \frac{1}{2}$, which shows that the formula (\ref{eq:NotWitBox}) is not valid in the model. 

On the other hand, this model is a $KG_\Box^\star$ model because it is trivially $\langle \Box, 0 \rangle$-witnessed\footnote{A model is $\langle \Box, 0 \rangle$-witnessed if $e(u, \Box \theta) = 0$ implies the existence of $v \in W$ such that $e(v, \theta) = 0 < Ruv$.}. Since it is known that $MG_\Box^\star$ is complete with respect to $0$-$\Box$-witnessed models \cite[p.~198]{CaRo10}, this concludes the proof.

\end{proof}

The previous implies that also $KG + \Box \neg \neg x \rightarrow \neg \neg \Box x$ and $MG + \neg \neg x \rightarrow \neg \neg \Box x$  are not complete w.r.t. ($\Box$) witnessed models.

The $\Diamond$-case requires a more delicate analysis, both in understanding why $UW_\Diamond$ fails to fully capture unwitnessed phenomena, and in identifying a formula that remains sound over finite models. Unlike the $\Box$-case, the formula $UW_\Diamond$ already addresses unwitnessing at intermediate truth values. However, a closer inspection reveals that it does so in a fundamentally insufficient way. Intuitively, $UW_\Diamond$ enforces that whenever $\Diamond x \leq \Diamond y$, there exists a world at which $x \leq y$. Yet this condition is too weak: it may be satisfied by worlds whose truth values play no essential role in determining the value of the $\Diamond$-formulas themselves.

This observation suggests the need for a more refined separation. In particular, by introducing a third variable $z$, one can isolate the values of $x$ and $y$ that are genuinely relevant for the evaluation of $\Diamond x$ and $\Diamond y$, distinguishing them from those that are not. Namely, $z$ allows us to separate the ``useful" values in computing the modal formulas. Using it, it is possible to state that the condition of $W_\Diamond$ is met, but specifically in the worlds really ``relevant" for $\Diamond x, \Diamond y$. The resulting formula can then be falsified in a model that still satisfies $UW_\Diamond$. The following lemma formalizes this idea.

\begin{lemma}\label{lem:NotWitDiamond}
$KG_\Diamond + (\Diamond x \rightarrow \Diamond y) \rightarrow (\neg \Diamond y \vee \Diamond(x \rightarrow y))$ is not complete with respect to witnessed models. In particular, it is not complete with respect to $\class{KG}_\Diamond$ witnessed models.
\end{lemma}
\begin{proof}

We will denote by $KG_\Diamond^\star$ to the logic $KG_\Diamond + (\Diamond x \rightarrow \Diamond y) \rightarrow (\neg \Diamond y \vee \Diamond(x \rightarrow y))$. 
Consider the formula 
\begin{equation}\label{eq:NotWitDiamond}
( (\Diamond x \rightarrow \Diamond y) \wedge ((\Diamond y \rightarrow \Diamond z) \rightarrow \Diamond z)) \rightarrow (\Diamond z \vee \Diamond ((x \rightarrow y) \wedge ((y \rightarrow z) \rightarrow z))).
\end{equation}
Let us first check that it is sound with respect to $\class{KG}_\Diamond$ witnessed models. Take a witnessed model $\mathfrak{M}$ and a world $u$ in it. We will check it satisfies the implication (\ref{eq:NotWitDiamond}) above, to which we will refer, in order to shorten it, by $A \rightarrow B$. Consider the different possible cases:
\begin{itemize}
\item If $e(u, \Diamond y) \leq e(u, \Diamond z)$, necessarily $e(u, A) \leq e(u, \Diamond z)$. Since $e(u, B) \geq e(u, \Diamond z)$, the case follows immediately.
\item Assume that ($*$): $e(u, \Diamond y) > e(u, \Diamond z)$. Then, $e(u, ((\Diamond y \rightarrow \Diamond z) \rightarrow \Diamond z)) = 1$. 
\begin{itemize}
\item If $(*.1)$: $e(u, \Diamond x) \leq e(u, \Diamond y)$, we have that $e(u, A) = 1$. Consider $v$ with $Ruv$ such that $e(u, \Diamond y) = e(v, y)$ (which exists because the model is witnessed). It follows from $(*.1)$ that $e(v, x) \leq e(v, y)$, and from ($*$), that $e(v,z) < e(v,y)$. This makes $e(v, (x \rightarrow y) \wedge ((y \rightarrow z) \rightarrow z)) = 1$, and so, $e(u, \Diamond ((x \rightarrow y) \wedge ((y \rightarrow z) \rightarrow z))) = 1$ proving the subcase.
\item Otherwise, if $(*.2)$: $e(u, \Diamond x) > e(u, \Diamond y)$, we have that $e(u, A) = e(u, \Diamond y)$. For $v$ with $Ruv$ such that $e(u, \Diamond y) = e(v, y)$ (as in the previous case),  we still have from (*) that $e(v,z) < e(v,y)$, so $e(v, ((y \rightarrow z) \rightarrow z)) = 1$. If $e(v, x) \leq e(v, y)$, we get that $e(v, (x \rightarrow y) \wedge ((y \rightarrow z) \rightarrow z)) = 1$, proving the subcase. On the other hand, if $e(v, x) > e(v,y)$, then $e(v, (x \rightarrow y) \wedge ((y \rightarrow z) \rightarrow z)) = e(v,y) = e(u, \Diamond y)$. Hence, $e(u, B) \geq e(u, \Diamond y)$, concluding that $e(u, A) \leq e(u,B)$. 
\end{itemize}
\end{itemize}

On the other hand, consider the model with universe
$\{u \cup \{v_i\colon i \in \omega\}\}, Ruv_i$ for all $i$ and valued as follows (again, values at $u$ are irrelevant):
\begin{align*}
e(v_1, z) = \frac{1}{3},&\qquad e(v_1, x) = \frac{1}{3},&\qquad e(v_1, y) = \frac{1}{3},&&  \\
e(v_i,z) = \frac{1}{3},&\qquad e(v_i, x) = \frac{1}{2}, &\qquad e(v_i, y) = \frac{1}{2} - \frac{1}{i+5}, && \text{ for }i > 1
\end{align*}

It is clear that $e(u, \Diamond z) = \frac{1}{3}$ and $e(u, \Diamond x) = e(u, \Diamond y ) = \frac{1}{2}$. Hence, $e(u, (\Diamond x \rightarrow \Diamond y) \wedge ((\Diamond y \rightarrow \Diamond z) \rightarrow \Diamond z)) = 1$. However, 
it is easy to see that $e(v_1, ((x \rightarrow y) \wedge ((y \rightarrow z) \rightarrow z))) = 1 \wedge \frac{1}{3} = \frac{1}{3} < 1$, and $e(v_i, ((x \rightarrow y) \wedge ((y \rightarrow z) \rightarrow z))) = \frac{1}{2} - \frac{1}{i+5} \wedge 1 = \frac{1}{2} - \frac{1}{i+5} < 1$, so $e(u,  (\Diamond z \vee \Diamond ((x \rightarrow y) \wedge ((y \rightarrow z) \rightarrow z))  = \frac{1}{3} \vee \frac{1}{2} = \frac{1}{2} < 1.$ This proves that the model is not a model of the formula (\ref{eq:NotWitDiamond}).

Showing that this structure is a model of $KG^*$ is more involved than in the previous case. On the one hand, no completeness result is currently known for the new logic. On the other hand, the formula~(\ref{eq:NotWitDiamond}) involves three variables (whereas the formula added to $KG_\Diamond$ involves only two), while in the $\Box$ case the formula~(\ref{eq:NotWitBox}) involves two variables, and the formula added to $KG_\Box$ involves only a single variable.

Our goal is to prove that for any formulas $\varphi, \psi$ in variables $x,y,z$, it holds that $e(u, \Diamond \varphi \rightarrow \Diamond \psi) \leq e(u, \neg \Diamond \psi) \vee e(u, \Diamond (\varphi \rightarrow \psi))$\footnote{We aim so prove the formula holds for any possible substitution}. While theoretically possible to do this by brute force (since G\"odel algebras are locally finite), it is known that the number of different (propositional) G\"odel formulas in three variables is greater than $10^{11}$ \cite{Ho69b}. Hence, we adopt a different approach. To prove our claim, we make use of the normal form for G\"odel propositional formulas (see, e.g., \cite{ZhaZha07}), which states that every G\"odel formula is equivalent to one in disjunctive normal form (DNF), i.e., a disjunction of conjunctions of basic terms of the form $a$, $a \rightarrow b$, and $(a \rightarrow b) \rightarrow b$, where $a,b$ are variables, together with the constant $\bot$.

Reasoning in such a way is enough -namely, using only propositional formulas- because formulas $\varphi, \psi$ will be only evaluated in worlds of the form $v_i$, and these do not have any successor. Thus, any subformula in them beginning with a modality can be substituted by $\top$ or $\bot$ respectively, and $\varphi$ and $\psi$ can be taken to be propositional formulas.

In order for formulas $\varphi, \psi$ to satisfy
\begin{equation}\label{eq:inequality}
    e(u, \Diamond \varphi \rightarrow \Diamond \psi) > e(u, \neg \Diamond \psi) \vee e(u, \Diamond (\varphi \rightarrow \psi))
\end{equation}
a first necessary condition is that $e(u, \neg \Diamond \psi) < 1$. Hence, $e(u, \Diamond \psi) > 0$, and so for some successor the value of $\psi$ is greater than $0$. It is also necessary that $e(u, \Diamond \varphi \rightarrow \Diamond \psi) > e(u, \Diamond (\varphi \rightarrow \psi))$. 
According to this, 
\begin{itemize}
    \item[(1)] $\Diamond \psi$ must be unwitnessed in $u$.
\end{itemize}
Otherwise, if $e(u, \Diamond \psi) = e(v, \psi)$ for some $v$, it would hold that $e(u, \Diamond \varphi \rightarrow \Diamond \psi) \leq e(v, \varphi \rightarrow \psi) \leq e(u, \Diamond (\varphi \rightarrow \psi))$).

As we pointed out above, any formula beginning with a modality is evaluated in any $v_i$ to $\bot$ or $\top$, so we can assume $\psi \coloneqq \bigvee_{i \in I} \psi_i$ (as in a G\"odel DNF). By the distributivity of $\Diamond$ and $\vee$, (1) implies that $e(u, \Diamond \psi) = e(u, \Diamond \psi_i)$ for some $i\in I$, and this latter one is unwitnessed. Also, we can assume that $\psi_i$ is a conjunction of basic G\"odel terms. Summing up, regarding $\psi$, we can limit our search to formulas with the structure $\bigwedge_{j \in J} t_j$, for $t_j$ basic G\"odel terms (in $x,y,z$) whose diamond is unwitnessed in $u$.  
 
Furthermore, to satisfy Equation (\ref{eq:inequality}) the following additional conditions must be met:
\begin{enumerate}
\item[(2)] $e(w, \varphi) > e(w, \psi)$ at any $w$ with $Ruw$ (otherwise, $e(u, \Diamond(\varphi \rightarrow \psi)) = 1$). This implies that, in particular, $e(w, \varphi) > e(w, \psi_i)$ for $\psi_i$ the formula above, and also, that  $e(u, \Diamond (\varphi \rightarrow \psi)) = e(u, \Diamond \psi) = e(u, \Diamond \psi_i)$, and
\item[(3)] $e(u, \Diamond \varphi) \leq e(u, \Diamond \psi) = e(u, \Diamond \psi_i)$.
\end{enumerate}

{\small 

\begin{table}[]
\makegapedcells
\begin{tabular}{llll | c llll}
Formula & $v_1$ & $v_i$ & equ. to & & Formula &  $v_1$ &  $v_i$ & equ. to  \\
  $  y   $ & $\frac{1}{3}$     &     $ \frac{1}{2}-\frac{1}{i+5} $&             & &$  \neg y$  & $0  $            &   $0$             &      $\bot$          \\
$\neg \neg y$&    $1$  &  $1$  &   $ \top$  & &     $y \rightarrow x$    &  $1$ &  $1$      & $\top$  \\
   $x$& $\frac{1}{3}$  &  $\frac{1}{2}$  &   & &    $\neg x$      &   $0$  &   $0$ &  $\bot$  \\ 
     $\neg \neg x$& $1$  &  $1$  &  $\top$ & &    $\neg \neg z$& $1$  &  $1$  &  $\top$ \\    
$z$& $\frac{1}{3}$  &  $\frac{1}{3}$  &   & &    $\neg z$      &   $0$  &   $0$ &  $\bot$\\
 $y \rightarrow z$& $1$  &  $\frac{1}{3}$  &      & &    $z \rightarrow y$     &   $1$  &   $1$ &   $\top$   \\
 $x \rightarrow z$ & $1$ & $\frac{1}{3}$ & $y \rightarrow z$    & &    $z \rightarrow x$     &   $1$  &   $1$ & $\top$   \\
  $x \rightarrow y$& $1$  &  $\frac{1}{2}-\frac{1}{i+5}$  &      & &    $(y \rightarrow x) \rightarrow x$     &   $\frac{1}{3}$  &   $\frac{1}{2}$ &  $x$  \\
 $(x \rightarrow y)\rightarrow y$& $\frac{1}{3}$  &  $1$  &      & &    $(z \rightarrow x) \rightarrow x$     &   $\frac{1}{3}$  &   $\frac{1}{2}$ & $x$   \\
 $(x \rightarrow z)\rightarrow z$& $\frac{1}{3}$  &  $1$  & $(x \rightarrow y)\rightarrow y$     &  &      &     &    &  \\ 
  $(z \rightarrow y)\rightarrow y$& $\frac{1}{3}$  &  $\frac{1}{2}-\frac{1}{i+5}$  &  $y$  & &    $(y \rightarrow z) \rightarrow z$     &   $\frac{1}{3}$  &   $1$ & $(x \rightarrow y) \rightarrow y$ \\ 
 
\end{tabular}\caption{Values of all basic G\"odel terms in $x,y,z$ in the model (Lemma \ref{lem:NotWitDiamond}.) }\label{table:formulas}
\end{table}
}

From Table \ref{table:formulas}, we can see that the only unwitnessed basic G\"odel term is $y$ (or is equivalent to $y$). By doing a similar exploration to the one in Table \ref{table:formulas}\footnote{Relying in the equivalences, to reduce the search.}, we can check that, in fact, there are no conjunctions of the basic G\"odel terms whose $\Diamond$ is unwitnessed in $u$, different from $y$ itself. 

Therefore, we may assume, without loss of generality, that the only possible formula $\psi$ that could make $e(u, \Diamond \varphi \rightarrow \Diamond \psi) > e(u, \neg \Diamond \psi) \vee e(u, \Diamond (\varphi \rightarrow \psi))$ is $y$. It remains to determine whether there exists a formula $\varphi$ that satisfies conditions $(2)$ and $(3)$ above (reading $\psi_i$ as $y$). 

In order to satisfy condition $(2)$, in particular, it is necessary that $e(v_1, \varphi) > e(v_1, y) =  \frac{1}{3}$. Therefore, necessarily, $e(v_1,\varphi) = 1$ (since there are no more possible values allowed in $v_1$). However, this forces $e(u, \Diamond \varphi) = 1$, which makes it impossible to satisfy condition (3) (i.e., $e(u, \Diamond \varphi) \leq e(u, \Diamond y) = \frac{1}{2}$), concluding the proof.

\end{proof}

\begin{remark}
    The formulas $(\Diamond x \rightarrow \Diamond y) \rightarrow (\neg \Diamond y \vee \Diamond(x\rightarrow y)$ and $((\Diamond x \rightarrow \Diamond y) \wedge ((\Diamond y \rightarrow \Diamond z) \rightarrow \Diamond z)) \rightarrow (\Diamond z \vee \Diamond ((x \rightarrow y) \wedge ((y \rightarrow z) \rightarrow z))$ are not valid over witnessed models with valued accessibility.  Consider the valued model $\langle W, R, e \rangle$ with $W = \{ v,w \}$, $R(v,y) = \frac{1}{2}$, and $e(w, x) = 1$, $e(w, y) = \frac{1}{2}$, $e(w, z) = 0$, so $e(v, \Diamond x) = \frac{1}{2}, e(v, \Diamond y) = \frac{1}{2}, e(v, \Diamond z) = 0$ and $e(v, \Diamond ((x \rightarrow y) \wedge ((y \rightarrow z) \rightarrow z) = \frac{1}{2}$. Therefore, in both cases, the premises of the scheme evaluate to $1$, but the consequence is $\frac{1}{2}$.\footnote{This counterexample is essentially the same as the one given in \cite[Pag. 202]{CaRo10}.}
\end{remark}

%


\section{Completeness with respect to partial witnessing conditions}\label{sec:partialWitnessing}

In this section we establish completeness results with respect to classes of models satisfying partial witnessing conditions for the $\Diamond$ modality. These results provide a refined semantic perspective that will play a central role in the subsequent axiomatization theorems. 

\begin{definition}
We say that a model is $\Diamond$-witnessed whenever for any formula $\Diamond \varphi$ and any $v \in W$, $e(v, \Diamond \varphi) = R(v,v_{\Diamond \varphi}) \wedge e(v_{\Diamond \varphi}, \varphi)$ for some $v_{\Diamond \varphi}$.\\
\noindent
We say that a \textit{crisp} model $\mathfrak{M}$ is $\langle \Diamond,1\rangle$-witnessed whenever, for any $\Diamond \varphi$ and any $v \in W$, if $e(v, \Diamond \varphi) = 1$ then there exists some $w \in W$ such that $Rvw$ and $e(w,\varphi) = 1$.
\end{definition}

We will prove that $MG$ is complete with respect to $\Diamond$-witnessed models (Theorem \ref{th:valueddiamondcompleteness}), and that $KG$ is complete with respect to $\langle \Diamond,1\rangle$-witnessed models (Theorem \ref{th:topwitnessed}).

Although these results arise naturally as intermediate steps toward the main axiomatization theorems of the following section, they are of independent conceptual significance. They fundamentally modify the semantic behavior of the $\Diamond$ modality, replacing a definition based on potentially unattained suprema with one grounded in effective witnesses. In the case of $MG$, the value of $\Diamond \varphi$ transitions from a possibly unattainable supremum to an actual maximum realized at some world. In the case of $KG$, we show that the top value $1$ is always attained. In both settings, this yields a substantial simplification in the semantic analysis of the logic and its extensions.

A first consequence of the theorems below is that introducing $\Diamond$-formulas as premises in a derivation guarantees the existence of worlds where these formulas are realized. This reflects a substantial change in the interpretation of the semantics. A particular example of this settles an open problem raised in \cite[Remark pag. 49]{CaRo15}: if the minimal bi-modal G\"odel logic is extended with the axiom $\Diamond \top$, then the underlying frame satisfies seriality, i.e., $\forall v \exists u : R(v,u) = 1$. 

Similarly, a direct adaptation of our methods yields analogous results for $K45$ and $KD45$ G\"odel logics \cite{BoEsGoRo16}. In these systems, the intended semantics is based on a unary accessibility relation $\Pi$, interpreted as a possibility measure, which in the case of $KD45$ is required to be normalized, i.e., $\bigvee_{v \in W} \Pi(v) = 1$. Our results imply completeness with respect to a stronger and more explicit semantics, in which there exists some $v \in W$ such that $\Pi(v) = 1$. Once again, this replaces a potentially unattainable supremum condition with an explicit existential one.

Another, more technical consequence concerns a deeper understanding of the relationship between the completeness results for $MG$ and $KG$. Recall that $KG = MG + {Cr}$ \cite{RoVi20}, where ${Cr}$ is Dunn's axiom. Our results clarify that even if the canonical model of $MG$ is forced to satisfy $Cr$, the resulting structure is not crisp. Indeed, since the canonical model of $MG$ is $\Diamond$-witnessed, crispness would transfer this property to $KG$, which is known not to be the case. It follows that the completeness proof for $KG$ necessarily relies on fundamentally different techniques from those used for $MG$. This distinction had previously been conjectured, but not formally established.

We proceed to prove the aforementioned completeness results.


\begin{theorem}\label{th:valueddiamondcompleteness}
$MG$ is complete with respect to the class of $\Diamond$-witnessed $\class{MG}$ models. In particular, the canonical model of $MG$ from \cite{CaRo15} is $\Diamond$-witnessed. 
\end{theorem}
\begin{proof}
This proof builds upon the definition of the canonical model in \cite{CaRo15}. This is defined, for a finite and closed (i.e., containing all of its subformulas) set of modal formulas $\Omega$, as the model with universe $W\coloneqq \{h \in Hom(Fm, [0,1]_G)\colon h(Th_{MG}) \subseteq \{1\}\},$ $e(h,p) \coloneqq h(p)$ for each propositional variable $p$, and $R(h,g) \coloneqq  \bigwedge\{(h(\Box \varphi) \rightarrow g(\varphi)) \wedge (g(\varphi) \rightarrow h(\Diamond \varphi))\colon \varphi \in \Omega\}$.
In \cite{CaRo15} the truth-lemma is proven (restricted to $\Omega$) for the above model, i.e., $e(h,\varphi) = h(\varphi)$ for every $\varphi \in \Omega$ and each $h$ in the universe. 

We will prove that, in fact, this model is also witnessed (restricted to $\Omega$) for formulas beginning with $\Diamond$, i.e., for each $v \in W$ and each $\Diamond \varphi \in \Omega$, there is $w \in W$ such that $R(v,w) \wedge e(w, \varphi) = e(v, \Diamond \varphi)$. This suffices to prove the theorem. 

Fix arbitrary $v$, $\Diamond \varphi$ as above, and assume $v(\Diamond \varphi) = \alpha >0$.\\ 
\textit{\textbf{Claim:}}
\[\varphi, \{\xi \colon \Box \xi \in \Omega, v(\Box \xi) \geq \alpha\} \not \vdash_{MG} \bigvee \{\theta \colon \Diamond \theta \in \Omega, v(\Diamond \theta) < \alpha\}\]
\textit{Proof of claim:}
\begin{quotation}
Assume the contrary. By D.T. and residuation,
\[\vdash_{MG} \varphi \rightarrow (\bigwedge \{\xi \colon \Box \xi \in \Omega, v(\Box \xi) \geq \alpha\} \rightarrow  \bigvee \{\theta \colon \Diamond \theta \in \Omega, v(\Diamond \theta) < \alpha\})\]
therefore, by $N_\Diamond$, 
\[\vdash_{MG} \Diamond \varphi \rightarrow \Diamond (\bigwedge \{\xi \colon \Box \xi \in \Omega, v(\Box \xi) \geq \alpha\} \rightarrow  \bigvee \{\theta \colon \Diamond \theta \in \Omega, v(\Diamond \theta) < \alpha\})\]
and by FS$_1$, 
\[\vdash_{MG} \Diamond \varphi \rightarrow (\Box \bigwedge \{\xi \colon \Box \xi \in \Omega, v(\Box \xi) \geq \alpha\} \rightarrow  \Diamond \bigvee \{\theta \colon \Diamond \theta \in \Omega, v(\Diamond \theta) < \alpha\}).\]
Lastly, by distributivity of $\Diamond$ over $\vee$ and of $\Box$ over $\wedge$, we would get that 
\[\vdash_{MG} \Diamond \varphi \rightarrow (\bigwedge \{\Box \xi \colon \Box \xi \in \Omega, v(\Box \xi) \geq \alpha\} \rightarrow  \bigvee \{\Diamond \theta \colon \Diamond \theta \in \Omega, v(\Diamond \theta) < \alpha\}).\]
But by assumption, $v(\Diamond \varphi) = \alpha$, and $v(\bigwedge \{\Box \xi \colon \Box \xi \in \Omega, v(\Box \xi) \geq \alpha\} \rightarrow  \bigvee \{\Diamond \theta \colon \Diamond \theta \in \Omega, v(\Diamond \theta) < \alpha\}) < \alpha$, leading to a contradiction. This concludes the proof of the claim. \qed
\end{quotation}

The above implies that there exists an element $u \in W$ such that (0) if $v(\Diamond \theta) < \alpha$ then $u(\theta) < 1$. In particular, 
\begin{enumerate}
\item[(1)] $u(\theta_1) < u(\theta_2)$ for any $\Diamond \theta_1, \Box \theta_2 \in \Omega$ such that $v(\Diamond \theta_1) < v(\Box \theta_2)$ and $v(\Diamond \theta_1) < \alpha$. This follows because from $\text{FS}_1$ we have that $v(\Diamond(\theta_2 \rightarrow \theta_1)) \leq v(\Box \theta_2 \rightarrow \Diamond \theta_1)$. From the assumptions, this equals $v(\Diamond \theta_1)$ and hence, is strictly less than $\alpha$,
\item[(2)] $u(\theta_1) \leq u(\theta_2)$ for any $\Diamond \theta_1, \Box \theta_2 \in \Omega$ such that $v(\Diamond \theta_1) \leq v(\Box \theta_2)$ and $v(\Diamond \theta_1) < \alpha$. This holds because $v(\Diamond((\theta_1 \rightarrow \theta_2) \rightarrow \theta_1)) \leq v(\Box(\theta_1 \rightarrow \theta_2) \rightarrow \Diamond \theta_1) \leq v((\Diamond \theta_1 \rightarrow \Box \theta_2) \rightarrow \Diamond \theta_1)$, where the first inequality follows from ${FS}_1$ and the second, from ${FS}_2$. From the assumptions, this equals $v(\Diamond \theta_1)$ and hence, is strictly less than $\alpha$. Therefore, $u((\theta_1 \rightarrow \theta_2) \rightarrow \theta_1) < 1$, so necessarily $u(\theta_1) \leq u(\theta_2)$.
\item[(3)]  $u(\theta) = 0 \Longrightarrow v(\Box \theta) = 0$, since $u(\theta) = u(\bot)$ implies by (2) above (contraposition), that $v(\Diamond \bot) \geq v(\Box \theta) $ or $v(\Diamond \bot) \geq \alpha$. This concludes the case, since $v(\Diamond \bot) = 0$ (axiom $F_\Diamond$). 
\item[(4)] $v(\Diamond \theta) = 0 \Longrightarrow u(\theta) = 0$, since in (3) above, $0 = v(\Diamond \theta) \leq v(\Box \bot)$ (hence $u(\theta) \leq 0$). 

\end{enumerate}

We now construct an object analogous to the one introduced in \cite{CaRo15}. 
Let $C=\{v(\Diamond \theta )\colon \Diamond \theta \in \Omega,  v(\Diamond \theta) \leq \alpha\}$ and define, for
each $c\in C$,
\begin{equation*}
u_{c}=\max \{u(\theta )\colon \theta \in \Omega,\ v(\Diamond \theta )=c\}.
\end{equation*}%
Note that $u_{0}=0$ by (4) above, and $u_{\alpha }=1$
because $u(\varphi )=1$. Define a finite increasing sequence $0=c_{0}<c_{1}<....$ of elements in $C$ as follows:\medskip
\begin{align*}
c_{0}\coloneqq& v(\Diamond \bot )=0\\
c_{i+1}\coloneqq& \min \{c\in C:c>c_{i}, u_{c}>u_{c_{i}}\}
\end{align*}

Let us denote by $\theta_i$ any formula in $\Omega$ such that $u_{c_i} = u(\theta_i)$ and $c_i = v(\Diamond \varphi_i)$. The last element of the above sequence is necessarily $c_N = \alpha$, since for any $c_i < \alpha$, we have that $c_i = v(\Diamond \varphi_i) < \alpha$ implies (by (0) above) that $u_{c_i} = v(\varphi_i) < 1$. Since $1 = u_{\alpha}$, necessarily there exists at least $c_{i+1} > c_i$. 

%

We further define (taking $\max \emptyset =0)$ a finite decreasing sequence 
\begin{align*}
q_{N-1}\coloneqq& \max \{v(\Box \theta )\colon v(\Box \theta )<c_{N}\},\\
q_{i-1}\coloneqq& \max \{c_{i},\max \{v(\Box \theta )\colon v(\Box \theta)<c_{i+1}\}\}, \text{ for } i\leq N-1.
\end{align*}

The previous sequences are such that:
\begin{align*}
0=c_{0}\leq q_{0}<c_{1}\leq q_{1}<....c_{N-1} \leq q_{N-1}<c_{N}=\alpha,\\
0=u_{c_{0}}<u_{c_{1}}<.....<u_{c_{N}}=1.
\end{align*}

Choose $g\colon [0,1]\rightarrow \lbrack 0,1]$ to be any strictly
increasing function such that
\begin{align*}
g(0)=0, &\qquad & g[(u_{c_{i}},u_{c_{i+1}}]]=&(q_{i},c_{i+1}] \text{ for } i<N-1, \\
g(1)=1, &\qquad  & g[(u_{c_{N-1}},1)]=&(q_{N-1},\alpha).
\end{align*}

It is immediate that $g$ is a G\"odel homomorphism and the valuation $w=g\circ v$
satisfies $w(\Sigma \cup T\mathcal{G}_{\Box \Diamond })=1$, hence
$w\in W.$ Furthermore, $w(\varphi) = 1 = w(\theta)$ for each $\theta$ with $v(\Box \theta) \geq \alpha$. 

We will prove now that $Rvw = \alpha$. First, we know that $w(\varphi) \rightarrow v(\Diamond \varphi) = 1 \rightarrow \alpha = \alpha$. It remains to verify that, for any $\Box \psi, \Diamond \theta \in \Omega$, 

$v(\Box \psi) \rightarrow w(\psi) \geq \alpha$ and
$w(\theta) \rightarrow v(\Diamond \theta) \geq \alpha$. 
\begin{itemize}
\item If $v(\Diamond \theta)\geq \alpha,$ clearly $w(\theta)\Rightarrow v(\Diamond \theta )\geq \alpha.$
\item If $v(\Diamond \theta )<\alpha $, we can prove that $w(\theta )\leq v(\Diamond \theta
).$  Consider two cases. First, if $u(\theta )\in
(u_{c_{i}},u_{c_{i+1}}) $ for some $i$ (recall that $u(\theta )<1$ by (0)). Then by construction of $g$, $w(\theta )\in (q_{i},c_{i+1}]$. As $u(\theta )>u_{c_{i}}$ and $c_{i+1}=v(\Diamond \varphi_{i+1})$ is the smallest $v(\Diamond \psi )$ with
$u(\psi)> u_{c_{i}} $ then $v(\Diamond \theta )\geq c_{i+1}\geq w(\theta ).$
Second, if $u(\theta )=0,$ then $w(\theta )=0$ and $v(\Box \theta )=0$ by (3).
\item  If $v(\Box \theta )\geq \alpha $ then $v(\Box \theta
)\Rightarrow w(\theta ) = 1 >\alpha$ because $u(\theta) = 1$.

\item $v(\Box \theta )<\alpha $ then $v(\Box \theta )\leq w(\theta ).$ To see this, note that $c_{i}\leq v(\Box \theta )\leq q_{i}<c_{i+1}$ for some $i$ and consider cases. First: $v(\Box \theta )=c_{i}=v(\Diamond \varphi _{i})$ then, by (2) $u_{c_{i}}=u(\varphi _{i})\leq u(\theta ).$ Therefore $c_{i}\leq w(\theta ).$ That is, $v(\Box \theta)\leq w(\theta ).$ Second: $c_{i}<v(\Box \theta )$ then $u_{c_{i}}<u(\theta )$, by (1) and by definition $q_{i}\leq w(\theta ),$ which shows again $v(\Box \theta )\leq w(\theta ).$
\end{itemize}

\end{proof}

\begin{theorem}\label{th:topwitnessed}
$KG$ (hence, also $KG_\Diamond$) is complete with respect to the class of $\langle \Diamond,1\rangle$-witnessed $\class{KG}$ models. In particular, the canonical model of $KG$ from \cite{RoVi20} is $\langle \Diamond,1 \rangle$-witnessed. 
\end{theorem}
\begin{proof}
In this case, we will work with the canonical model for $KG$ given in \cite{RoVi20}. For an arbitrary $\rho$, this is $\mathfrak{M}^\Omega$ for $\Omega =SFm(\rho)$, defined as the model with universe $W\coloneqq \{h \in Hom(Fm, [0,1]_G)\colon h(Th_{KG}) \subseteq \{1\}\},$ $e(h,p) \coloneqq h(p)$ for each propositional variable $p$, and $Rhg$ if and only if for any $\Box \varphi \in \Omega$ $h(\Box \varphi) \leq g(\varphi)$ and for any $\Diamond \varphi \in \Omega$, $g(\varphi) \leq h(\Diamond \varphi)$. 

We will prove that this model is $\langle \Diamond,\top\rangle $-witnessed.

Consider a  world $h$ in it and a formula $\Diamond \psi \in \Omega$ such that $h(\Diamond \psi) = 1$.
Necessarily, there is some $g \in W$ such that $g(\chi) < g(\psi)$ for each $\Diamond \chi \in SFm( \rho)$ with $h(\Diamond \chi) < 1$. 

Define the extended set of variables $\mathcal{V}^+ \coloneqq \mathcal{V}ar(\Omega) \cup \{\star \psi \in \Omega\colon \star \in \{\Box, \Diamond\}\}$. 
Let $g' \colon \mathcal{V}^+ \rightarrow [0,1]$ be defined by

\[g'(p) \coloneqq \begin{cases} g(p) &\hbox{ if }g(p) < g(\psi),\\ 1 &\hbox{ otherwise.} \end{cases}\]

A straightforward calculation shows that\footnote{Recall that we need to check this for propositional connectives only, since formulas beginning by a modality are interpreted by fresh variables. It is important that $\&$ equals $\wedge$, otherwise this claim would not necessarily hold.} to see that this extends to formulas, namely, that for any (modal) formula $\varphi$ \[g'(\varphi) = \begin{cases} g(\varphi) &\hbox{ if } g(\varphi) < g(\psi),\\ 1 &\hbox{ otherwise.} \end{cases}\]

Since $g(x) \leq g'(x)$ for any $x$, in particular $g'(\theta) = 1$ for any $\theta \in Th(KG)$, so $g'$ is an element of $W^\Omega$. We will check that $Rhg'$, i.e., by definition of the canonical model, that $h(\Box \varphi) \leq g(\varphi)$ for any $\Box \varphi \in \Omega$, and that $g(\varphi) \leq h(\Diamond \varphi)$ for any $\Diamond \varphi \in \Omega$. 
It is first clear that $h(\Box \varphi) \leq g'(\varphi)$, since $g(\varphi) \leq g'(\varphi)$ for any formula, and $h(\Box \varphi) \leq g(\varphi)$ due to the fact that we know that $Rhg$.

Regarding the $\Diamond$ condition, assume $h(\Diamond \varphi) < 1$ (otherwise, the case is trivial). Therefore, we know that $h(\Diamond \varphi) < h(\Diamond \psi)$, and so, from the choice of $g$ in the beginning of the proof, we know that $g(\varphi) < g(\psi)$. Therefore, $g'(\varphi) = g(\varphi) \leq h(\Diamond \varphi)$, where the last inequality follows from $Rhg$. 

With this, we have proven that $Rhg'$, and since $g'$ evaluates $\psi$ at the top, this concludes the proof.

\end{proof}



\section{Axiomatizations of the logics of finite G\"odel-Kripke models}\label{sec:axiomatization{}}

We now show that the formulas introduced in Lemmas \ref{lem:NotWitBox} and \ref{lem:NotWitDiamond} suffice to obtain a complete axiomatization with respect to witnessed models. This yields a full axiomatic characterization of the logics under consideration. Since these logics are complete with respect to models of finite modal depth (see Observation \ref{obs:finDepthDecidable}), it follows in particular that we obtain an axiomatization of the local consequence induced by finite G\"odel–Kripke models.

Let us name
\begin{align*}
W_\Box \coloneqq& \Box (((x \rightarrow y) \rightarrow y) \wedge (y \rightarrow x)) \rightarrow ((\Box x \rightarrow \Box y) \rightarrow \Box y)\\
W_\Diamond \coloneqq& ( (\Diamond x \rightarrow \Diamond y) \wedge ((\Diamond y \rightarrow \Diamond z) \rightarrow \Diamond z)) \rightarrow (\Diamond z \vee \Diamond ((x \rightarrow y) \wedge ((y \rightarrow z) \rightarrow z)))
\end{align*}



\begin{theorem}\label{th:completenessCrisp}
The logic $KG + W_\Box + W_\Diamond$ is complete with respect to $\class{KG}^w$. Moreover, $KG_\Box + W_\Box$, and $KG_\Diamond + W_\Diamond$ are also complete with respect to $\class{KG}^w$.\footnote{Recall that we allow for the same class of models to be complete with respect to different modal logics, varying their modal language.}
\end{theorem}
\begin{proof}
We saw in Lemma  \ref{lem:NotWitDiamond} that the formula $W_\Diamond$ is sound with respect to $\Diamond$-witnessed models, and in Lemma \ref{lem:NotWitBox} that $W_\Box$ is sound with respect to $\Box$-witnessed models.

To prove completeness, consider the definition of the canonical model $\mathfrak{M}^\Omega$ from \cite{RoVi20}\footnote{For the interested reader, see the definition in the proof of Theorem \ref{th:topwitnessed}}, and modify it by considering the restricted universe 
\[W^\Omega_{wit} \coloneqq \{u \in Hom(Fm, [0, 1]_G): u(Th(KG + W_\Box + W_\Diamond)) = \{1\}\}\]
The fact that this model satisfies the truth lemma for the formulas in $\Omega$ is as in \cite{RoVi20}, and it is $\langle \Diamond, 1 \rangle$-witnessed by Theorem \ref{th:topwitnessed}.

We will now check that this model is witnessed for the formulas in $\Omega$, which will conclude the proof of the theorem.

Let us begin by the $\Box$-case. Consider any world $h$ and any formula $\Box \varphi \in \Omega$, and let $\alpha \coloneqq h(\Box \varphi)$ and let $\beta \coloneqq max\{h(\Box \psi)\colon \Box \psi \in \Omega, h(\Box \psi) < \alpha\}$ if this maximum exists, and $0$ otherwise. Observe that, in this latter case, also $\alpha= 0$, since otherwise, $h(\Box \bot) = 0 < \alpha$.
 Assume also that $\alpha < 1$, since otherwise the formula is trivially witnessed. Let 
\[
\chi \coloneqq \bigwedge \{\psi \colon \Box \psi \in \Omega, h(\Box \psi) > \alpha\} \qquad \text{ and } \qquad
\phi \coloneqq \bigwedge \{\psi \colon \Box \psi \in \Omega, h(\Box \psi) = \alpha\}.
\]

It is easy to check that 
\begin{equation}\label{eq:notTheoremBox}
\Box(\chi \rightarrow  ((\varphi \rightarrow \phi )\rightarrow \phi) \wedge (\phi \rightarrow \varphi))
\end{equation}
is not a theorem of the logic. 
Otherwise, by $K_\Box$ we would have that $\Box \chi \rightarrow \Box( (\varphi \rightarrow \phi )\rightarrow \phi) \wedge (\phi \rightarrow \varphi))$ is a theorem of the logic. Therefore, through $W_\Box$, we would get that 
$\Box \chi \rightarrow ((\Box \varphi \rightarrow \Box \phi) \rightarrow \Box \phi)$. But this is a contradiction, since 
$h(\Box \chi) > \alpha$ while $h((\Box \varphi \rightarrow \Box \phi) \rightarrow \Box \phi) = \alpha$ (using distributivity of $\Box$ and $\wedge$). 

The fact that formula (\ref{eq:notTheoremBox}) is not a theorem, implies that there exists some world $g$ in the canonical model such that $Rhg$ and (*) $g(\chi) > g(((\varphi \rightarrow \phi) \rightarrow \phi) \wedge (\phi \rightarrow \varphi))$. Moreover, we can assume $g(\varphi) > \alpha$ (which is also either strictly greater than $\beta$, or both $\alpha= \beta = 0$), since otherwise $g$ is the witness we are looking for. 

Note that 
\[g(((\varphi \rightarrow \phi) \rightarrow \phi) \wedge (\phi \rightarrow \varphi)) = \begin{cases} 1 &\hbox{ if } g(\phi) < g(\varphi) \text{ or } g(\phi) = g(\varphi) = 1,\\ g(\phi) &\hbox{ if } e(\varphi) \leq e(\phi) < 1,\\ g(\varphi) &\hbox{ if } e(\phi) > e(\varphi)\end{cases}\]

In any case, this value is greater than or equal to $g(\varphi)$, hence from $(*)$ it follows that $g(\chi) > g(\varphi)$, and also that $e(\varphi) \leq e(\phi)$ (otherwise, the value of the formula is $1$).

If $\beta < \alpha$, consider the endomorphism $\sigma \colon [0,1] \rightarrow [0,1]$ defined by

\[\sigma(x) \coloneqq \begin{cases} x &\hbox{ if }x > g(\varphi),  \\ \frac{ x - \beta}{g(\varphi) - \beta}(\alpha - \beta) + \beta &\hbox{ if } x \in (\beta, g(\varphi)], \\ x &\hbox{ otherwise.}\end{cases}\]

It is well defined since $g(\varphi) > \alpha$ (and so, greater than $\beta)$.

Otherwise (i.e., if $\beta = \alpha = 0)$, define simply 
\[\sigma(x) \coloneqq \begin{cases} x &\hbox{ if }x > g(\varphi),  \\ 0 &\hbox{ otherwise.}\end{cases}\]

In both cases, we have that $g' \coloneqq \sigma \circ g$ is a homomorphism of $Fm$ into $[0,1]_G$, and such that $g(\psi) = 1$ implies that $g'(\psi) = 1$ too, so $g' \in W^{\Omega}_{wit}$. Let us check that $Rhg'$, namely that $h(\Box \varphi) \leq g'(\varphi)$ for any $\Box \varphi \in \Omega$, and that $g'(\varphi) \leq h(\Diamond \varphi)$ for any $\Diamond \varphi \in \Omega$.
Since $\sigma(a) \leq a$ for any $a$, we have that $g'(x) \leq g(x)$ for any $x$, so it is clear that $h(\Diamond \psi) \geq g(\psi) \geq g'(\psi)$ for any $\psi$.

To check the $\Box$ cases, let us consider the three possibilities:
\begin{itemize}
\item If $h(\Box \psi) > \alpha$, then we know that $\psi$ is an element of the conjunction of the definition of $\chi$, so $g(\psi) \geq g(\chi)$. Above we proved that $g(\chi) > g(\varphi)$, and hence, by the definition of $\sigma$, we get $\sigma(g(\psi)) = g(\psi)$, concluding that $h(\Box \psi) \leq g'(\psi)$. 
\item If $h(\Box \psi) =  \alpha$, then $\psi$ is an element of the conjunction of the definition of $\phi$, so $g(\psi) \geq g(\phi)$. Above we proved that $g(\phi) \geq g(\varphi)$, and hence, by the definition of $\sigma$, we get $\sigma(g(\psi)) \geq g(\varphi) = \alpha$. Therefore, $h(\Box \psi) \leq g'( \psi)$.
\item If $h(\Box \psi) < \alpha$ (therefore, $\beta < \alpha$, and we are in the first definition of $\sigma$), we know by the definition of $\beta$ that $h(\Box \psi) \leq \beta$. Now, if $g(\psi) \leq \beta$ the endomorphism $\sigma$ behaves as the identity, so $h(\Box \psi) \leq g'(\psi)$. Otherwise, namely, if $g(\psi) > \beta$, we know that $\sigma(g(\psi)) > \beta$ too. Therefore, $h(\Box \psi) \leq \beta < g'(\psi)$, proving the case. 
\end{itemize}

Let us now check that the model is also witnessed for $\Diamond$ formulas, using only $W_\Diamond$.

Dually to the $\Box$, consider any world $h$ and any formula $\Diamond \varphi \in \Omega$, and let $\alpha \coloneqq h(\Diamond \varphi)$. Assume also that $\alpha > 0$ since otherwise, again, the case is trivial, and also that $\alpha < 1$, since otherwise we know the model is $\langle \Diamond,1\rangle$-witnessed by Lemma \ref{th:topwitnessed}. Therefore, we let $\beta \coloneqq \min\{h(\Diamond \psi) \colon \Diamond \psi \in \Omega, h(\Diamond \psi) > h(\Diamond \varphi)\}$, and observe this minimum always exists, since $h(\Diamond \top) = \top > \alpha$.

Let
\[
\chi \coloneqq \bigvee \{\psi \colon \Diamond \psi \in \Omega, h(\Diamond \psi) < \alpha\} \qquad \text{ and }\qquad
\phi \coloneqq \bigvee \{\psi \colon \Diamond \psi \in \Omega, h(\Diamond \psi) = \alpha \}.
\]

Since $h(W_\Diamond) = 1$, in particular substituting $\phi, \varphi$ and $\chi$ we have that 
\[h(((\Diamond \phi \rightarrow \Diamond \varphi) \wedge ((\Diamond \varphi \rightarrow \Diamond \chi) \rightarrow \Diamond \chi)) \rightarrow (\Diamond \chi \vee \Diamond ((\phi \rightarrow \varphi) \wedge ((\varphi \rightarrow \chi) \rightarrow \chi)))) =1.\]

By the definitions of $\phi$ and $\chi$, and the distributivity of $\Diamond$ and $\vee$ 
\[h(((\Diamond \phi \rightarrow \Diamond \varphi) \wedge ((\Diamond \varphi \rightarrow \Diamond \chi) \rightarrow \Diamond \chi))) = 1.\]
Therefore, it follows that also
\[h(\Diamond \chi \vee \Diamond ((\phi \rightarrow \varphi) \wedge ((\varphi \rightarrow \chi) \rightarrow \chi))) = 1.\]

By definition $h(\Diamond \chi) < \alpha < 1$, therefore, $h(\Diamond ((\phi \rightarrow \varphi) \wedge ((\varphi \rightarrow \chi) \rightarrow \chi))) = 1$. Hence from Lemma \ref{th:topwitnessed} follows that there exists a world $g$ such that $Rhg$ and $e(g, (\phi \rightarrow \varphi) \wedge ((\varphi \rightarrow \chi) \rightarrow \chi))= 1$. Thus, for this world $g$, it holds that  $e(g, \phi) \leq e(g, \varphi)$ and $e(g, \chi) < e(g, \varphi)$ (since, by construction, $e(g, \chi) \neq 1)$).

Consider then the endomorphism $\sigma \colon [0,1] \rightarrow [0,1]$ defined by

\[\sigma(x) \coloneqq \begin{cases} x &\hbox{ if }x < e(g, \varphi),  \\ \frac{ \beta - x}{\beta - e(g, \varphi)}(\beta - \alpha) + \alpha &\hbox{ if } x \in [e(g, \varphi), \beta), \\ x &\hbox{ otherwise.}\end{cases}\]

We have that $g' \coloneqq \sigma \circ g$ is an homomorphism of $Fm$ into $[0,1]_G$, and such that $g(\psi) = 1$ implies that $g'(\psi) = 1$ too, so $g' \in W^\Omega_{wit}$. Let us check that $Rhg'$. 
Since $g'(x) \geq g(x)$, it is clear that $h(\Box \psi) \leq g(\psi) \leq g'(\psi)$ for any $\psi$.

To check the $\Diamond$ cases, let us consider the three possibilities:
\begin{itemize}
\item If $h(\Diamond \psi) < \alpha$, then we know $\psi$ is an element of the disjunction of the definition of $\chi$, so $e(g, \psi) \leq e(g, \chi)$. Above we saw that $e(g, \chi) < e(g, \varphi)$, and hence, by the definition of $\sigma$, we get that $\sigma(e(g, \psi)) = e(g, \psi)$, concluding that $h(\Diamond \psi) \geq g'(\psi) = e(g, \psi)$. 

\item If $h(\Diamond \psi) = \alpha$, then $\psi$ is an element of the disjunction of the definition of $\phi$, so $e(g, \psi) \leq e(g, \phi)$. Above we saw that $e(g, \phi) \leq e(g, \varphi)$, and hence, by the definition of $\sigma$, we get that $\sigma(e(g, \psi)) \leq e(g, \varphi) = \alpha$. Therefore, $h(\Diamond \psi) \geq g'(\psi)$.

\item If $h(\Diamond \psi) > \alpha$, we know that $h(\Diamond \psi) \geq \beta$. Now, if $g(\psi) \geq \beta$ the endomorphism $\sigma$ acts as the identity, so $h(\Box \psi) \geq g'(\psi)$. Otherwise, namely, if $g(\psi) < \beta$, we know that $\sigma(e(g, \psi)) < \beta$ too. Therefore, $h(\Diamond \psi) \geq \beta > g'(\psi)$, proving the case. 
\end{itemize}

\end{proof}

\begin{corollary}
  $KG + W_\Box + W_\Diamond$ is complete with respect to the class of finite $\class{KG}$ models, and the same holds for the corresponding mono-modal extensions.
\end{corollary}

\begin{theorem}\label{th:completenessVal}
$MG + W_\Box$ is complete with respect to $\class{MG}^w$. 
\end{theorem}
\begin{proof}
 The proof is very similar to the one for $KG$, with one minor difference that we will point out.

As before, consider the definition of the canonical model $\mathfrak{M}^\Omega$ (now from \cite{CaRo10}, see the definition in the proof of Theorem \ref{th:valueddiamondcompleteness}) with the modified universe
\[W^\Omega_{wit} \coloneqq \{u \in Hom(L_{\Box\Diamond}(V), [0, 1]_G): u(Th(MG + W_\Box)) = \{1\}\}\]
The fact that this model satisfies the truth lemma for the formulas in $\Omega$ is proved as in \cite{CaRo10}, and again, it follows from Theorem \ref{th:valueddiamondcompleteness} that this model is $\Diamond$-witnessed. Therefore, we only need to check the witnessing property for the formulas beginning with $\Box$ belonging to $\Omega$.  Since we proved in Lemma \ref{lem:NotWitBox} that $W_\Box$ is sound with respect to $\class{MG}^w$ this will conclude the proof of the theorem.

Again as before, consider any world $h$ and any formula $\Box \varphi \in \Omega$, and let $\alpha \coloneqq h(\Box \varphi)$ and $\beta \coloneqq max\{h(\Box \psi)\colon \Box \psi \in \Omega, h(\Box \psi) < \alpha\}$ if this maximum exists and $0$ otherwise -again, this implies $\alpha= 0$ too.
 Assume also that $\alpha < 1$, since otherwise the formula is trivially witnessed. Let 
\[
\chi \coloneqq \bigwedge \{\psi \colon \Box \psi \in \Omega, h(\Box \psi) > \alpha\}\qquad \text{ and } \qquad
\phi \coloneqq \bigwedge \{\psi \colon \Box \psi \in \Omega, h(\Box \psi) = \alpha\}.
\]
Again, for the same reasons as in the $KG$ case, 
\[\Box(\chi \rightarrow  ((\varphi \rightarrow \phi )\rightarrow \phi) \wedge (\phi \rightarrow \varphi))\] is not a theorem of the logic. 

In the $MG$ case, this now implies that  there exists some world $g$ such that $R(h,g) >  g(\chi \rightarrow  ((\varphi \rightarrow \phi )\rightarrow \phi) \wedge (\phi \rightarrow \varphi))$ (this is the main difference with the $KG$ case). 
Nevertheless, as in the $KG$ case, this implies that  (*) $g(\chi) > g(((\varphi \rightarrow \phi) \rightarrow \phi) \wedge (\phi \rightarrow \varphi))$.
As before, the value of the right side is greater than or equal to $g(\varphi)$, so from $(*)$ it follows that $g(\chi) > g(\varphi)$, and also that $e(\varphi) \leq e(\phi)$. Moreover, since $R(h,g) >  g(\chi \rightarrow  ((\varphi \rightarrow \phi )\rightarrow \phi) \wedge (\phi \rightarrow \varphi))$, necessarily 
$R(h,g) > g(\varphi)$ too. 

Assume $g(\varphi) > \alpha$, since otherwise $g$ is the witness we are looking for. As in the $KG$ case, for $\beta < \alpha$ consider then the endomorphism $\sigma \colon [0,1] \rightarrow [0,1]$ defined by

\[\sigma(x) \coloneqq \begin{cases} x &\hbox{ if }x > g(\varphi),  \\ \frac{ x - \beta}{g(\varphi) - \beta}(\alpha - \beta) + \beta &\hbox{ if } x \in (\beta, g(\varphi)], \\ x &\hbox{ otherwise.}\end{cases}\]

It is well defined since $g(\varphi) > \alpha$ (and so, greater than $\beta)$.

Otherwise (i.e., if $\beta = \alpha = 0$) let, as before, 

\[\sigma(x) \coloneqq \begin{cases} x &\hbox{ if }x > g(\varphi),\\ 0 &\hbox{ otherwise.}\end{cases}\]

We have that $g' \coloneqq \sigma \circ g$ is an homomorphism of $Fm$ into $[0,1]_G$, and such that $g(\psi) = 1$ implies that $g'(\psi) = 1$ too, so $g' \in W^{\Omega}_{wit}$. Let us check that $Rhg' \geq Rhg > \alpha$ (the main difference regarding calculations with respect to the crisp case).

Recall that $Rhg \coloneqq \bigwedge_{\Box \chi \in \Omega}h(\Box \chi) \rightarrow g(\chi) \wedge \bigwedge_{\Diamond \chi \in \Omega}g(\chi) \rightarrow h(\Diamond \chi)$.
Since $g'(x) \leq g(x)$, it is clear that $g'(\psi) \rightarrow h(\Diamond \psi) \geq g(\psi) \rightarrow h(\Diamond \psi)$ for any $\psi$.

To check the $\Box$ cases, let us consider the three possibilities:
\begin{itemize}
\item If $h(\Box \psi) > \alpha$, as in the $KG$ case it follows that $\sigma(g(\psi)) = g(\psi)$, concluding that $h(\Box \psi) \rightarrow g'(\psi) =  h(\Box \psi) \rightarrow g(\psi)$. 
\item If $h(\Box \psi) = g(\varphi) = \alpha$, as in the $KG$ case, we get that $\sigma(g(\psi)) \geq g(\varphi) = \alpha$. Therefore, $h(\Box \psi) = \alpha \leq g'(\psi)$.
\item If $h(\Box \psi) < g(\varphi)$ -therefore, $\beta < \alpha$ and we are using the first definition of $\sigma-$, again as in the $KG$ case,  if $g(\psi) \leq \beta$ the endomorphism $\sigma$ behaves as the identity, so $h(\Box \psi) \rightarrow g'(\psi) = h(\Box \psi) \rightarrow g(\psi)$, and otherwise, $h( \Box \psi) \leq \beta < g'(\psi)$.
\end{itemize}

\end{proof}

\begin{corollary}
$MG + W_\Diamond$ is complete with respect to the class of finite $\class{MG}$ models.
\end{corollary}

A natural observation that follows from the previous results is that  $KG+W_\Box+W_\Diamond$, and $MG+W_\Box$ are the corresponding logics arising from the class of models evaluated over all finite G\"odel chains, and equivalently, over the G\"odel set $[0,1]_{\uparrow} = \{0, \frac{1}{2},\frac{2}{3},\ldots\frac{i}{i+1},\ldots 1\}$. It is also natural to see that the logics arising from the models evaluated over $[0,1]_{\downarrow}$ are axiomatized as the ones before, but substituting $W_\Box$ by $W_\Box' \coloneqq \neg \Box x \vee W_\Box$.




\section{Conclusions}

In this paper, we have provided a comprehensive analysis of the completeness of modal G\"odel logics with respect to finite valued Kripke models. We solved a long-standing open problem by demonstrating that the previously proposed candidate formulas, $UW_\Box$ and $UW_\Diamond$, are insufficient to restore completeness for finite G\"odel-Kripke models. Building on these negative results, we successfully introduced new unwitnessing formulas and established sound and complete axiomatizations for both the crisp and valued finite-model semantics, including their mono-modal fragments. Furthermore, we refined the semantic landscape by characterizing intermediate partial witnessing conditions that hold for basic modal G\"odel logics.

\section{Acknowledgements}
Amanda Vidal has been funded by the ERDF-Project "Knowledge in the Age of Distrust" from the Programme Johannes Amos Comenius under the Ministry of Education, Youth and Sports of the Czech Republic, CZ.02.01.01/00/23\_025/0008711. Ricardo O. Rodriguez would like to remark that this work was completed despite the lack of support from the scientific funding agencies of Argentina. On the other hand, his work has been funded by the European MSCA-SE SemPER project No. 101299559.

\bibliographystyle{plain}

\end{document}